\documentclass[11pt,twoside,a4paper]{article}
\usepackage{amsmath,amssymb,amsthm}
\usepackage{hyperref}
\usepackage{graphicx,psfrag}

\newtheorem{thm}{Theorem}[section]
\newtheorem{lem}[thm]{Lemma}
\newtheorem{pro}[thm]{Proposition}
\newtheorem{cor}[thm]{Corollary}

\theoremstyle{definition}
\newtheorem{exa}[thm]{Example}

\theoremstyle{remark}

\newcommand{\R}{\mathbb{R}}
\newcommand{\Z}{\mathbb{Z}}
\newcommand{\N}{\mathbb{N}}

\newcommand{\cI}{\mathcal{I}}

\newcommand{\cP}{\mathcal{P}}

\newcommand{\de}{\delta}
\newcommand{\De}{\Delta}
\newcommand{\ep}{\varepsilon}
\newcommand{\om}{\omega}

\newcommand{\si}{\sigma}

\renewcommand{\phi}{\varphi}

\newcommand{\CAT}{\operatorname{CAT}}
\newcommand{\hyp}{\operatorname{H}}

\newcommand{\pr}{\operatorname{pr}}
\newcommand{\prr}{\operatorname{prr}}

\newcommand{\crr}{\operatorname{cr}}
\newcommand{\intr}{\operatorname{int}}

\newcommand{\reg}{\operatorname{reg}}

\newcommand{\sign}{\operatorname{sign}}
\newcommand{\ay}{\operatorname{aY}}
\newcommand{\harm}{\operatorname{Harm}}
\newcommand{\hm}{\operatorname{Hm}}

\newcommand{\lh}{\operatorname{lh}}
\newcommand{\rh}{\operatorname{rh}}
\newcommand{\h}{\operatorname{h}}
\newcommand{\co}{\operatorname{co}}

\newcommand{\es}{\emptyset}
\newcommand{\set}[2]{\{#1:\,\text{#2}\}}
\newcommand{\sm}{\setminus}
\newcommand{\sub}{\subset}

\newcommand{\wt}{\widetilde}
\newcommand{\wh}{\widehat}

\begin{document}

\title{On the inverse problem of M\"obius geometry on the circle}
\author{Sergei Buyalo\footnote{Supported by RFFI Grant
17-01-00128a}}

\date{}
\maketitle

 \begin{abstract} Any (boundary continuous) hyperbolic space induces on the 
 boundary at infinity a M\"obius structure which reflects most
 essential asymptotic properties of the space. In this paper, we initiate
 the study of the inverse problem: describe M\"obius structures which are
 induced by hyperbolic spaces at least in the simplest case of the circle.
 For a large class of M\"obius structures on the circle, we define a canonical
 ``filling'' each of them, which serves as a natural candidate for a solution of the 
 inverse problem. This is a 3-dimensional (pseudo)metric space
 $\harm$, 
 which consists of harmonic 4-tuples of the respective M\"obius structure with a distance 
 determined by zig-zag paths. Our main result is the proof that every line in
 $\harm$
 is a geodesic, i.e., shortest in the zig-zag distance on each segment. This gives a good 
 starting point to show that 
 $\harm$
 is Gromov hyperbolic with the prescribed M\"obius structure at infinity.
 \end{abstract}

\noindent{\small{\bf Keywords:} M\"obius structures, cross-ratio, harmonic 4-tuples}

\medskip

\noindent{\small{\bf Mathematics Subject Classification:} 51B10}

\section{Introduction} A M\"obius structure on a set 
$X$
is a class of (semi)metrics whose cross-ratios take one and the same value on every given
4-tuple of points in
$X$.
M\"obius structures naturally arise as geometric structures on the boundary
at infinity of hyperbolic spaces. The classical example is the extended
Euclidean space
$\wh\R^n=\R^n\cup\{\infty\}$,
which gives rise to the canonical M\"obius structure 
$M_0$
over the sphere
$S^n=\wh\R^n$,
whose group of M\"obius transformations is isomorphic to the isometry
group of the hyperbolic space
$\hyp^{n+1}$.

The inverse problem of M\"obius geometry asks to describe M\"obius structures which are
induced by hyperbolic spaces. The papers \cite{BS14}, \cite{BS15} can be regarded as
solutions of this problem in the case of rank 1 symmetric spaces. In a general case, it seems
very little is known, cp.~\cite{BeS17}, \cite{BFI18}. Thus we consider a simplest nontrivial case when
$X=S^1$
is the circle.

The class of all M\"obius structures on the circle is very large: any extended (semi)metric
on
$\wh\R$
generates some M\"obius structure on
$S^1$.
Note that various hyperbolic cone constructions (see \cite{BoS}, \cite{BS07}) give a hyperbolic metric space with
prescribed metric at infinity. However, no one of them is equivariant with respect to M\"obius
transformations of the metric. Thus one can consider the inverse problem as the existence problem
of an equivariant hyperbolic cone over a given metric.

Asking more, one should pay an additional price for that: we introduce a set of axioms,
which allow to define a reasonable candidate for a solution of the inverse problem.
This is the set 
$\harm$
of harmonic 4-tuples with respect to a given M\"obius structure
$M$. 
It has a natural structure
of a 3-dimensional manifold, which in the case of the canonical structure 
$M_0$
is homeomorphic
to the projectivized tangent bundle of
$\hyp^2$.
Note that
$\harm$
is automatically invariant under M\"obius transformations of 
$M$.

It follows from our axioms that any pair 
$(x,y)$
of different points in
$X$
uniquely determines a line 
$h=h_{(x,u)}$
in
$\harm$,
which consists of all pairs of different points
$(z,u)$
such that 4-tuple
$q=((x,y),(z,u))$
is harmonic. It turns out that
$h$
is homeomorphic to
$\R$
and, moreover,
$h$
is isometric to
$\R$
with respect to the naturally defined distance
$$|qq'|=\left|\ln\frac{d(x,z')d(y,z)}{d(x,z)d(y,z')}\right|,$$
$q'=((x,y),(z',u'))$,
where 
$d$
is any metric from
$M$
($|qq'|$
is independent of the choice of
$d$).

The pairs
$(x,y)$, $(z,u)$
are called {\em axes} of
$q\in\harm$.
Since every harmonic
$q$
has two axes, moving along a line in
$\harm$,
there is a possibility to change the axis at any moment. This leads to a notion
of special curves in
$\harm$,
which are called {\em zz-paths}. Every (finite) zz-path
$\si\sub\harm$
consists of a finite number of consecutive sides, every side is a segment of a line,
and adjacent sides meet each other at a common harmonic 4-tuple 
$q$
as the different
axes of
$q$.
The point of this construction is that while in general two different
$q$, $q'\in\harm$
cannot be connected by a segment of a line, they are always connected by 
a finite zz-path.

The length of a zz-path
$\si$
is the sum of the lengths of its sides
$|\si|$.
The 
$\de$-distance
on
$\harm$
is defined by
$$\de(q,q')=\inf_\si|\si|,$$
where the infimum is taken over all zz-paths between 
$q$
and
$q'$.
The 
$\de$-distance
is symmetric, nonnegative and satisfies the triangle inequality.
However, it is not clear that
$\de$
is positive, i.e.,
$\de$
is a pseudometric. Nevertheless, our main result says that lines are geodesics 
with respect to the 
$\de$-distance.

\begin{thm}\label{thm:main} Every line
$h\sub\harm$
is a geodesic with respect to the 
$\de$-distance, 
i.e.
$\de(q,q')=|qq'|$
for any
$q$, $q'\in h$. 
\end{thm}

This is not at all obvious or trivial. The precise statement of Theorem~\ref{thm:main}
requires to list axioms for M\"obius structures under which the theorem is true,
see sect.~\ref{sect:distance_segments}. The key property we require from a M\"obius structure 
to satisfy Theorem~\ref{thm:main} is the {\em Increment} axiom,
see sect.~\ref{subsect:increment_axiom}. To prove Theorem~\ref{thm:main}, for every
line
$h\sub\harm$
we define so called {\em midpoint projection} of
$\harm$
to
$h$.
The increment axiom allows to show that the midpoint projection decreases distances along 
zz-paths, which leads to Theorem~\ref{thm:main}.

{\it Aknowledgment.} The author is very much grateful to Viktor Schroeder for numerous
discussions on the topic of the paper, which lead, in particular, to the notion of 
a monotone M\"obius structure, and for the proof of Lemma~\ref{lem:unique_common_perpendicular}.

\section{M\"obius structures}
\label{sect:moebius_structures}

\subsection{Basic notions}
\label{subsect:basics}

Let
$X$
be a set. A 4-tuple
$q=(x,y,z,u)\in X^4$
is said to be {\em admissible} if no entry occurs three or
four times in
$q$.
A 4-tuple
$q$
is {\em nondegenerate}, if all its entries are pairwise
distinct. Let
$\cP_4=\cP_4(X)$
be the set of all ordered admissible 4-tuples of
$X$, $\reg\cP_4\sub\cP_4$
the set of nondegenerate 4-tuples.

A function
$d:X^2\to\wh\R=\R\cup\{\infty\}$
is called a {\em semi-metric}, if it is symmetric,
$d(x,y)=d(y,x)$
for each
$x$, $y\in X$,
positive outside of the diagonal, vanishes on the diagonal
and there is at most one infinitely remote point
$\om\in X$
for
$d$,
i.e. such that
$d(x,\om)=\infty$
for some
$x\in X\sm\{\om\}$.
Moreover, we require that if
$\om\in X$
is such a point, then
$d(x,\om)=\infty$
for all 
$x\in X$, $x\neq\om$.
A metric is a semi-metric that satisfies the triangle inequality.

A {\em M\"obius structure}
$M$
on
$X$
is a class of M\"obius equivalent semi-metrics on
$X$,
where two semi-metrics are equivalent if and only if they have
the same cross-ratios on every
$q\in\reg\cP_4$.

Given
$\om\in X$,
there is a semi-metric 
$d_\om\in M$
with infinitely remote point
$\om$.
It can be obtained from any semi-metric
$d\in M$
for which 
$\om$
is not infinitely remote by a {\em metric inversion},
$$d_\om(x,y)=\frac{d(x,y)}{d(x,\om)d(y,\om)}.$$
Such a semi-metric is unique up to a homothety, see \cite{FS},
and we use notation
$|xy|_\om=d_\om(x,y)$
for the distance between
$x$, $y\in X$
in that semi-metric. We also use notation
$X_\om=X\sm\{\om\}$.

There is a distinguished class of M\"obius structures called {\em ptolemaic}.
The property to be ptolemaic is characterized by the inequality

\begin{equation}\label{eq:ptolemaic}
d(x,y)d(z,u)\le d(x,z)d(y,u)+d(x,u)d(y,z)
\end{equation}
for every semi-metric
$d$
of the M\"obius structure and every 4-tuple
$q=(x,y,z,u)\in X^4$.
The property to be ptolemaic is invariant under any metric inversion, and this
invariance can serve as an equivalent definition of ptolemaic M\"obius structures.
It follows from (\ref{eq:ptolemaic}) that any semi-metric of a ptolemaic M\"obius
structure with infinitely remote point
$\om\in X$
is a metric on
$X_\om$,
i.e., it satisfies the triangle inequality.

Every M\"obius structure
$M$
on
$X$
determines the 
$M$-{\em topology}
whose subbase is given by all open balls centered at finite points
of all semi-metrics from
$M$
having infinitely remote points.

\begin{exa}\label{exa:canonical_moebius_circle} Our basic example is the 
{\em canonical} M\"obius structure 
$M_0$
on the circle
$X=S^1$.
We think of
$S^1$
as the unit circle in the plane,
$S^1=\set{(x,y)\in\R^2}{$x^2+y^2=1$}$.
For 
$\om=(0,1)\in X$
the stereographic projection
$X_\om\to\R$
identifies
$X_\om$
with real numbers 
$\R$.
We let
$d_\om$
be the standard metric on
$\R$,
that is,
$d_\om(x,y)=|x-y|$
for any
$x,y\in\R$.
This generates a M\"obius structure on
$X$
which is called {\em canonical}. The basic feature of the canonical M\"obius 
structure on
$X=S^1$
is that for any 4-tuple
$(\si,x,y,z)\sub X$
with the cyclic order 
$\si xyz$
we have 
$d_\si(x,y)+d_\si(y,z)=d_\si(x,z)$.
In particular, the canonical M\"obius structure is ptolemaic.
\end{exa}

\subsection{An alternative description}
\label{subsect:alternative}

The following is an alternative description of a M\"obius structure which
is convenient in many cases. For any semi-metric
$d$
on
$X$
we have three cross-ratios
$$q\mapsto \crr_1(q)=\frac{|x_1x_3||x_2x_4|}{|x_1x_4||x_2x_3|};
  \crr_2(q)=\frac{|x_1x_4||x_2x_3|}{|x_1x_2||x_3x_4|};
  \crr_3(q)=\frac{|x_1x_2||x_3x_4|}{|x_2x_4||x_1x_3|}$$
for 
$q=(x_1,x_2,x_3,x_4)\in\reg\cP_4$,
whose product equals 1, where
$|x_ix_j|=d(x_i,x_j)$.
We associate with 
$d$
a map 
$M_d:\reg\cP_4\to L_4$
defined by
\begin{equation}\label{eq:moeb_map}
M_d(q)=(\ln\crr_1(q),\ln\crr_2(q),\ln\crr_3(q)),
\end{equation}
where
$L_4\sub\R^3$
is the 2-plane given by the equation
$a+b+c=0$.
Two semi-metrics
$d$, $d'$
on
$X$
are M\"obius equivalent if and only
$M_d=M_{d'}$.
Thus a M\"obius structure on
$X$
is completely determined by a map 
$M=M_d$
for any semi-metric
$d$
of the M\"obius structure, and we often identify a M\"obius structure
with the respective map 
$M$.

Let 
$S_n$
be the symmetry group of 
$n$
elements. The group
$S_4$
acts on
$\reg\cP_4$
by entries permutations of any
$q\in\reg\cP_4$.
The group
$S_3$
acts on
$L_4$
by signed permutations of coordinates, where a permutation
$\si:L_4\to L_4$
has the sign 
``$-1$'' 
if and only if
$\si$
is odd. 

The {\em cross-ratio} homomorphism
$\phi:S_4\to S_3$
can be described as follows: a permutation of a tetrahedron
ordered vertices 
$(1,2,3,4)$
gives rise to a permutation of pairs of opposite edges
$((12)(34),(13)(24),(14)(23))$. 
We denote by 
$\sign:S_4\to\{\pm 1\}$
the homomorphism that associates to every odd permutation the sign 
``$-1$''.

One easily check that any M\"obius structure 
$M:\reg\cP_4\to L_4$
is equivariant with respect to the signed cross-ratio homomorphism,
\begin{equation}\label{eq:signed_cross-ratio_homomorphism}
M(\pi(q))=\sign(\pi)\phi(\pi)M(q) 
\end{equation}
for every
$q\in\reg\cP_4$, $\pi\in S_4$,
where
$\phi:S_4\to S_3$
is the cross-ratio homomorphism.

\subsection{Monotone M\"obius structures}

In what follows we assume that M\"obius structures we consider are ptolemaic.
We say that a M\"obius structure
$M$
on
$X=S^1$
is {\em monotone}, if it satisfies the following axioms
\begin{itemize}
 \item [(T)] Topology: $M$-topology
on
$X$
is that of
$S^1$; 
\item[(M)] Monotonicity: given a 4-tuple
$q=(x,y,z,u)\in X^4$
such that the pairs
$(x,y)$, $(z,u)$
separate each other,
we have
$$|xy|\cdot|zu|>\max\{|xz|\cdot|yu|,|xu|\cdot|yz|\}$$
for some and hence any semi-metric from
$M$.
\end{itemize}

A choice of
$\om\in X$
uniquely determines the interval
$xy\sub X_\om$
for any distinct
$x$, $y\in X$
different from
$\om$
as the arc in
$X$
with the end points
$x$, $y$
that does not contain
$\om$.
As an useful reformulation of Axiom~(M) we have

\begin{cor}\label{cor:interval_monotone} Assume for a nondegenerate
4-tuple
$q=(x,y,z,u)\in\reg\cP_4$
the interval
$xz\sub X_u$
is contained in
$xy$, $xz\sub xy\sub X_u$.
Then
$|xz|_u<|xy|_u$.
\end{cor}

\begin{proof} By the assumption, the pairs
$(x,y)$, $(z,u)$
separate each other. Hence, by Axiom~(M) we have
$|xz||yu|<|xy||zu|$
for any semi-metric from
$M$.
In particular,
$|xz|_u<|xy|_u$.
\end{proof}

\begin{lem}\label{lem:nozero_value} Assume a M\"obius structure
$M$
on
$X=S^1$
is monotone. Then
$M(q)\neq(0,0,0)$
for every
$q\in\reg\cP_4$.
\end{lem}

\begin{proof} Assume
$M(q)=(0,0,0)$
for 
$q=(x,y,z,u)\in\reg\cP_4$.
Then in a metric from
$M$
with infinitely remote point
$u$
we have
$|xy|_u=|xz|_u=|yz|_u$.
Whatever is the order of
$x,y,z$
on
$X_u=X\sm\{u\}$,
these equalities contradict the mononicity Axiom~(M). 
\end{proof}

\subsection{Increment axiom}
\label{subsect:increment_axiom}

Increment axiom for monotone M\"obius structures
has been introduced in \cite{Bu17}, where
it plays an important role since it implies the time 
inequality. In this paper, it also plays a key role in solving the 
inverse problem for M\"obius structures on the circle. We briefly
recall this axiom and some properties of monotone M\"obius structures
satisfying it.

We use notation
$\reg\cP_n$
for the set of ordered nondegenerate
$n$-tuples
of points in
$X=S^1$, $n\in\N$.
For 
$q\in\reg\cP_n$
and a proper subset
$I\sub\{1,\dots,n\}$
we denote by
$q_I\in\reg\cP_k$, $k=n-|I|$,
the 
$k$-tuple
obtained from
$q$
(with the induced order) by crossing out all entries which correspond to elements of
$I$.

(I) Increment Axiom: for any 
$q\in\reg\cP_7$
with cyclic order
$\co(q)=1234567$
such that 
$q_{247}$
and 
$q_{157}$
are harmonic, we have 
$$\crr_1(q_{345})>\crr_1(q_{123}).$$

For definition of harmonic 4-tuples see sect.~\ref{subsect:harmonic_4_tuples}.
It is proved in \cite[Proposition~7.10]{Bu17} that the canonical M\"obius
structure
$M_0$
on the circle
$X=S^1$
satisfies Increment Axiom. Moreover, the class
$\cI$
of monotone M\"obius structures on the circle which satisfy Axiom~(I)
contains an open in a fine topology neighborhood of 
$M_0$,
see \cite[Proposition~7.14]{Bu17}.

\section{Filling}
\label{sect:filling}

Here we define a space of harmonic pairs which will serve as a filling 
of a monotone M\"obius structure on the circle.

\subsection{Harmonic 4-tuples}
\label{subsect:harmonic_4_tuples}

Let
$M$
be a monotone M\"obius structure on the circle
$X=S^1$.
A 4-tuple
$q\in\reg\cP_4$
is said to be {\em harmonic} if
$M(q)\in L_4$
has a zero coordinate. It follows from Lemma~\ref{lem:nozero_value} for
$q$
harmonic,
$M(q)$
has a unique zero coordinate. Therefore, we have three types of harmonic
4-tuples
$q=(x,y,z,u)\in\reg\cP_4$,
determined by conditions
\begin{itemize}
 \item [(1)] $|xz|\cdot|yu|=|xu|\cdot|yz|$,
\item[(2)]  $|xu|\cdot|yz|=|xy|\cdot|zu|$,
\item[(3)] $|xy|\cdot|zu|=|xz|\cdot|yu|$,
\end{itemize}
for some and hence every semi-metric from
$M$,
which correspond to the first, the second and the third coordinate of
$M(q)$
respectively.

\begin{lem}\label{lem:three_embeddings} For
$i=1,2,3$
there is an embedding
$e_i:\reg\cP_3\to\reg\cP_4$
of the set 
$\reg\cP_3\sub X^3$
of nondegenerate 3-tuples, whose image
$e_i(\reg\cP_3)$
is the set of harmonic 4-tuples of type 
$(i)$.
\end{lem}

\begin{proof} Given
$t=(x_1,x_2,x_3)\in\reg\cP_3$.
we take a semi-metric
$|\cdot\cdot|_i$
from
$M$
with infinitely remote point
$x_i$.
The distance function
$x\mapsto|x_{i+1}x|_i$
is continuous on
$X_{x_i}$
(see \cite[Lemma~4.1]{Bu17}), thus there is
$y_i\in X_{x_i}$
with
$|x_{i+1}y_i|_i=|y_ix_{i+2}|_i$
(indices are taken modulo 3). By Corollary~\ref{cor:interval_monotone},
$y_i$
is uniquely determined and moreover the pairs
$(x_i,y_i)$
and
$(x_{i+1},x_{i+2})$
separate each other. Now, we put
$e_i(t)=(y_i,x_1,x_2,x_3)$.
By constuction,
$e_i(t)$
satisfies
$$|x_{i+1}y_i|\cdot|x_ix_{i+2}|=|y_ix_{i+2}|\cdot|x_ix_{i+1}|$$
for any semi-metric from
$M$,
and thus 
$e_i(t)$
is harmonic of type
$(i)$.
Conversely, given a harmonic 4-tuple
$q=(x,y,z,u)$,
we take either of
$y$, $z$, $u$
as an infinitely remote point and see that
$x$
is the midpoint between remaining two ones for harmonicity type (1), (2), (3)
respectively. Therefore, every harmonic 4-tuple of type 
$(i)$ 
is
$e_i(t)$
for an appropriate
$t\in\reg\cP_3$.
\end{proof}

The set 
$\reg\cP_3\sub X^3$
in the induced topology consists of two connected components each of which
is homeomorphic to the unit tangent bundle
$U\hyp^2$
of the hyperbolic plane
$\hyp^2$,
that is, it is the trivial 
$S^1$-bundle
over 
$\R^2$.

By Lemma~\ref{lem:three_embeddings}, 
$e_i(\reg\cP_3)$
is the set of harmonic 4-tuples of type 
$(i)$.
Therefore, the set of harmonic 4-tuples consists of six connected
components each of which is homeomorphic to
$\R^2\times S^1$.
The group
$S_4$
acting on
$\reg\cP_4$
permutes these components with the stabilizer of each one isomorphic to the 
cyclic group
$\Z_4$.
These facts are not used in what follows, they only describe the general structure
of the space of harmonic 4-tuples.

\subsection{Harmonic pairs}
\label{subsect:harm_pairs}

As a topological space, the required filling is defined as the set 
$\harm$
of harmonic pairs. It is convenient to use unordered pairs 
$(x,y)\sim(y,x)$
of distinct points on
$X=S^1$,
and we denote the set of them by
$\ay=S^1\times S^1\sm\De/\sim$,
where
$\De=\set{(x,x)}{$x\in S^1$}$
is the diagonal. A pair 
$(a,b)\in\ay\times\ay$
is harmonic if
\begin{equation}\label{eq:harmonic}
|xz|\cdot|yu|=|xu|\cdot|yz|
\end{equation}
for some and hence any semi-metric of the M\"obius structure, where
$a=(x,y)$, $b=(z,u)$.
That is, we use the first type of harmonic 4-tuples to define harmonic
pairs. The choice of the type is irrelevant to our construction
because different types of harmonicity are permuted with each other by
the group
$S_4$.

Note that the pairs of points
$a$, $b$
separate each other for every harmonic pairs
$(a,b)$.
This follows from mononicity of
$M$,
see the proof of Lemma~\ref{lem:three_embeddings}.

The set 
$\harm$
of the harmonic pairs is a 3-dimensional subspace in
$\ay\times\ay$
given by Equation~(\ref{eq:harmonic}). There is an involution
$\pi(x,y,z,u)=(y,x,u,z)$
acting on the set of harmonic 4-tuples of the first type which factors that
set to
$\harm$.
Therefore,
$\harm$
is homeomorphic the projectivized tangent bundle of
$\hyp^2$.

Given
$q=(a,b)\in\harm$,
the pair
$a\in\ay$
is called the {\em left axis} and the pair
$b\in\ay$
the {\em right axis} of 
$q$.

There is a canonical involution 
$j:\harm\to\harm$
without fixed points given by
$j(a,b)=(b,a)$.
The quotient space we denote by
$\hm:=\harm/j$.
In other words,
$\hm$
is the set of unordered harmonic pairs of unordered pairs 
of points in
$X$.
Note that
$j(q)=(b,a)$
is harmonic with the left axis
$b$
and the right axis
$a$
for every harmonic pair
$q=(a,b)\in\harm$.

The space
$\harm$
has two canonical structures of a locally trivial bundle
$\pr_i:\harm\to\ay$
with respect to the factor projections
$\pr_i:\ay\times\ay\to\ay$, $i=1,2$,
$\pr_1(a,b)=a$, $\pr_2(a,b)=b$.
It follows from Lemma~\ref{lem:three_embeddings}, that the fibers of
$\pr_i$
are homeomorphic to an open arc in
$S^1$,
i.e. to
$\R$.
We obviously have
$\pr_i\circ j=\pr_{i+1}$
for 
$i=1,2$,
where the indices are taken modulo 2. Both
$\R$-bundles
$\pr_1$, $\pr_2$
are nontrivial, i.e. 
$\harm$
is not homeomorphic to the product
$\ay\times\R$.

\subsection{Lines and zig-zag paths in $\harm$}
\label{subsect:lines_harm}

A {\em left line} 
$\lh_a$, $a\in\ay$,
in
$\harm$
is the subset
$\lh_a=\pr_1^{-1}(a)\sub\harm$.
The pair
$a\in\ay$
is called the {\em axis} of
$\lh_a$.
Similarly, a {\em right line}
$\rh_b$, $b\in\ay$,
is the subset
$\rh_b=\pr_2^{-1}(b)\sub\harm$.
The pair
$b\in\ay$
is called the {\em axis} of
$\rh_b$.
Note that
$j(\lh_a)=\rh_a$
and
$j(\rh_b)=\lh_b$.

Every fiber of the fibration 
$\pr_1:\harm\to\ay$
is a left line, while every fiber of the fibration
$\pr_2:\harm\to\ay$
is a right line. Thus every left (right) line is homeomorphic to
$\R$.
The axis 
$a$
of
$\lh_a$
is the common left axis for all
$q\in\lh_a$.
The axis 
$b$
of
$\rh_b$
is the common right axis for all
$q\in\rh_b$.

A line in
$\hm$
is the image of a left line or a right line under the canonical projection
$\harm\to\hm$.
Thus in
$\hm$
we do not distinguish left and right lines. The notion of the axis of 
a line is preserved by
$j$,
and we denote by
$\h_a\sub\hm$
a line with the axis
$a\in\ay$.

We say that 
$b$, $b'\in\ay$
are in the {\em strong causal relation} if either of them lies
on an open arc in
$X$
determined by the other one (more for this terminology see in \cite{Bu17}).

\begin{lem}\label{lem:common_perp} For different
$q=(a,b)$, $q'=(a,b')$
lying of on a left line
$\lh_a$,
the pairs
$b$, $b'\in\ay$
are in the strong causal relation. Conversely, given
$b$, $b'\in\ay$
in the strong causal relation, there is a left line
$\lh_a$
such that 
$q=(a,b)$, $q'=(a,b')\in\lh_a$.
Similar properties hold true also for right lines and lines in
$\hm$.
\end{lem}

\begin{proof} The arguments can be found in \cite[Proposition~5.8, Proposition~3.2(b)]{Bu17}.
For convenience of the reader we briefly recall them.
 
Let
$a=(x,y)$, $b=(z,u)$, $b'=(z',u')\in\ay$,
where
$q=(a,b)$, $q'=(a,b')$
lie on a left line
$\lh_a$.
Taking a semi-metric from
$M$
with infinitely remote point
$x$,
we observe that
$y$
is the midpoint of the segments
$zu$, $z'u'\sub X_x$.
Since
$b\not=b'$,
we can assume that 
$z'y\sub zy$.
By Axiom~(M),
$|z'y|_x<|zy|_x$,
and thus
$|u'y|_x<|uy|_x$.
Then again by Axiom~(M),
$u'y\sub uy$.
It follows that
$b'$
lies on an open arc in
$X$
determined by
$b$,
i.e.,
$b$, $b'$
are in the strong causal relation.

Conversely, Lemma~\ref{lem:three_embeddings} implies that for every
$b=(z,u)\in\ay$
there is a well defined involutive homeomorphism
$\rho_b:X\to X$, 
called the {\em reflection} with respect to
$b$,
that fixes
$z$, $u$,
such that the pair
$(a,b)$
is harmonic for every
$x\in X\sm b$,
where
$a=(x,\rho_b(x))$.
For
$b$, $b'\in\ay$
in the strong causal relation, we take the composition
$\rho=\rho_b\circ\rho_{b'}$
of the respective reflection and note that 
$\rho(b^+)\sub\intr(b^+)$,
where
$b^+\sub X$
is the closed arc determined by
$b$
that does not include
$b'$.
Thus there is a fixed point
$x\in\intr b^+$
of
$\rho$.
Then
$a=(x,y)\in\ay$, 
where
$y=\rho_{b'}(x)$,
is preserved by
$\rho_b$, $\rho_{b'}$,
and 
$q=(a,b)$, $q=(a,b')\in\lh_a$.
\end{proof}
The pair
$a\in\ay$
above is called a {\em common perpendicular} to
$b$, $b'$.
We postpone the proof of uniqueness to sect.~\ref{subsect:distance_harmonic_pairs},
see Lemma~\ref{lem:unique_common_perpendicular}.

We say that
$d\in\ay$
{\em separates}
$b$
and
$c\in\ay$
if
$b$
and
$c$
lie on different open arcs in
$X$
defined by
$d$.
Note that in this case
$b$, $c$, $d$
are in the strong causal relation with each other.

Given a left line
$\lh_a\sub\harm$
and distinct
$q=(a,b)$, $q'=(a,b')\in\lh_a$,
we define the {\em left segment}
$qq'\sub\lh_a$
as the union of 
$q$, $q'$
and all of
$q''=(a,b'')\in\lh_a$
such that 
$b''$
separates
$b$, $b'$.
The points
$q$, $q'$
are the {\em ends} of
$qq'$.
Similarly, we define right segments on a right line. More generally,
a segment 
$qq'$
in
$\harm$ ($\hm$)
is a segment of line in
$\harm$ ($\hm$).
In this case the harmonic pairs
$q$, $q'$
have a common axis.

By the first part of Lemma~\ref{lem:common_perp},
$b$, $b'$
are in the strong causal relation. Denote by
$b^-\sub X$
the open arc determined by
$b$
that contains
$b'$,
and by
$(b')^-$
the open arc determined by
$b'$
that contains
$b$.
Then
$a$
does not meet
$b^-\cap(b')^-$
because 
$a,b$ 
separate each other as well as
$a,b'$.
By Lemma~\ref{lem:three_embeddings}, for every
$z''\in b^-\cap(b')^-$
there is
$u''\in X$
such that
$(a,b'')$
is harmonic, i.e.,
$(a,b'')\in\lh_a$,
where
$b''=(z'',u'')$.
Thus
$b''$
is in the strong causal relation with 
$b$
as well as with
$b'$.
Hence,
$u''\in b^-\cap(b')^-$.
In other words, the intersection
$b^-\cap(b')^-$
is invariant under the reflection
$\rho_a:X\to X$,
see proof of Lemma~\ref{lem:common_perp}.
We conclude that the segment
$qq'\sub\lh_a$
is homeomorphic to the standard segment
$[0,1]$.

A {\em zig-zag} path, or zz-path, 
$S\sub\harm$
is defined as an alternating finite (maybe empty) sequence  of left and right segments 
$\si_i$
in
$\harm$,
where consecutive segments 
$\si_i$, $\si_{i+1}$
have a common end. Segments 
$\si_i$
are also called {\em sides} of
$S$.

\begin{lem}\label{lem:zz_connected} Given
$q$, $q'\in\harm$,
there is a zz-path
$S$
in
$\harm$
with at most five sides that connects
$q$
and
$q'$.
\end{lem}

\begin{proof} Let
$q=(a,b)$, $q'=(a',b')$.
The pairs
$a$, $a'\in\ay$
separate
$X$
into (at most four) open arcs. Taking
$a''\in\ay$
on such an arc, we see that
$a''$
is in the strong causal relation with
$a$
as well as with
$a'$.
By Lemma~\ref{lem:common_perp}, there is a common
perpendicular
$\wt b$
to
$a$, $a''$,
and there is a common perpendicular
$\wt b'$
to
$a'$, $a''$.
Then the pairs
$\wt q=(a,\wt b)$, $q''=(a'',\wt b)$,
$\wt q''=(a'',\wt b')$, $\wt q'=(a',\wt b')$ 
are harmonic, and the alternating sequence
$$S=q\wt q,\ \wt q q'',\  q''\wt q'',\ \wt q''\wt q',\ \wt q'q'$$
of left and rigth segments connects
$q$, $q'$
having at most 5 sides.
\end{proof}

A zz-path in
$\hm$
is the image of a zz-path in
$\harm$
under the canonical projection
$\harm\to\hm$.
This is also an alternating (in obvious sence) finite sequence
of segments in
$\hm$,
where consecutive segments have a common end.
Lemma~\ref{lem:zz_connected} holds true also in
$\hm$.

\section{Pseudometric on $\harm$}
\label{sect:pseudometric}

\subsection{Distance between harmonic pairs with common axis}
\label{subsect:distance_harmonic_pairs}

Given two harmonic pairs in
$q$, $q'\in\harm$
with a common axis, say
$q=(a,b)$
and
$q'=(a,b')$,
we define {\em the distance}
$|qq'|$
between them as
\begin{equation}\label{eq:distance}
|qq'|=|j(q)j(q')|=\left|\ln\frac{|xz'|\cdot|yz|}{|xz|\cdot|yz'|}\right|
\end{equation}
for some and hence any semi-metric on
$X$
from
$M$,
where
$a=(x,y)$, $b=(z,u)$, $b'=(z',u')\in\ay$,
and
$j:\harm\to\harm$
is the canonical involution.
Note that
\begin{equation}\label{eq:distance_different}
|qq'|=\left|\ln\frac{|xu'|\cdot|yu|}{|xu|\cdot|yu'|}\right|=
             \left|\ln\frac{|xu'|\cdot|yz|}{|xz|\cdot|yu'|}\right|=
             \left|\ln\frac{|xz'|\cdot|yu|}{|xu|\cdot|yz'|}\right|
\end{equation}
by harmonicity of
$q$, $q'$.
In this way, (\ref{eq:distance}) defines the distance along
the left hyperbolic line
$\lh_a\sub\harm$
as well as along the right hyperbolic line
$\rh_a\sub\harm$.

\begin{lem}\label{lem:distance_additive} Given 
$a\in\ay$
and
$q=(a,b)$, $q'=(a,b')$, $q''=(a,b'')\in\lh_a$
such that
$b'$
separates
$b$
and
$b''$,
then 
$|qq''|=|qq'|+|q'q''|$.
A similar property holds true also for right lines.
\end{lem}

\begin{proof} Let
$a=(x,y)$, $b=(z,u)$, $b'=(z',u')$, $b''=(z'',u'')$.
In the semi-metric from
$M$
with infinitely remote point
$x$, $y$
is the midpoint of the segments
$zu$, $z'u'$, $z''u''\sub X_x$.
Using that
$b'$
separates
$b$
and
$b''$,
we can assume without loss of generality that
$z''u''\sub z'u'\sub zu$.
Then
$|yz''|_x<|yz'|_x<|yz|_x$
and thus
$$|qq'|=\ln\frac{|yz|_x}{|yz'|_x},\ |q'q''|=\ln\frac{|yz'|_x}{|yz''|_x},
\ |qq''|=\ln\frac{|yz|_x}{|yz''|_x}.$$
Therefore
$|qq''|=|qq'|+|q'q
''|$.
\end{proof}

Now, we can prove uniqueness of the common perpendicular.

\begin{lem}\label{lem:unique_common_perpendicular} Given
$b$, $b'\in\ay$
in the strong causal relation, there is at most one common perpendicular
$a\in\ay$
to
$b$, $b'$.
\end{lem}

\begin{proof} Assume there are common perpendiculars
$a=(z,u)$, $a'=(z',u')\in\ay$
to
$b$, $b'$. 
By the first part of Lemma~\ref{lem:common_perp}, 
$a$
and 
$a'$
are in the strong causal relation.
Let
$b=(x,y)$, $b'=(x',y')$.
Using that the pairs
$a,a'$
and 
$b,b'$
are in the strong causal relation, we assume without loss of generality
that on 
$X_x$
we have the following order of points
$zz'yy'u'ux'$.
We denote by
$q_1=(b,a)$, $q_2=(b',a)$, 
$q_1'=(b,a')$, $q_2'=(b',a')$
respective harmonic pairs. Then
$q_1,q_2\in\rh_a$, $q_1',q_2'\in\rh_{a'}$,
and we have well defined distances
$l=|q_1q_2|$, $l'=|q_1'q_2'|$.
Computing them in a semi-metric of the M\"obius
structure with infinitely remote point
$x$,
we obtain
$$e^l=\frac{|zx'|}{|x'u|}\quad
  e^{l'}=\frac{|z'x'|}{|x'u'|}.$$
Using the order of points
$zz'yy'u'ux'$
on
$X_x$,
we have, in particular, that the interval
$z'x'$
is contained in the interval
$zx'$.
By Corollary~\ref{cor:interval_monotone},
$|zx'|\ge|z'x'|$.
Similarly,
$x'u\sub x'u'$
and hence
$|x'u|\le|x'u'|$.
Thus
$l\ge l'$
and if 
$a'\neq a$,
the inequality is strong. Applying this argument with infinitely remote point
$y$,
we obtain
$l\le l'$.
Therefore
$l=l'$
and
$a=a'$. 
\end{proof}

\subsection{Defining a pseudometric metric $\de$ on $\hm$ and $\harm$}
\label{subsect:def_pseudometric}

It follows from Lemma~\ref{lem:distance_additive}, that for harmonic pairs
$q$, $q'$
on one and the same left (right) line, the length of the segment
$\si=qq'$
is equal to the distance
$|qq'|$,
that is, it can be computed by any of Equalities~(\ref{eq:distance}), 
(\ref{eq:distance_different}).

Let
$S=\{\si_i\}$
be a zz-path in
$\harm$.
We define length of
$S$
as the sum
$|S|=\sum_i|\si_i|$
of the length of its sides. Now, we define a distance
$\de$
on
$\harm$
by
$$\de(q,q')=\inf_S|S|,$$
where the infimum is taken over all zz-paths
$S\sub\harm$
from
$q$
to
$q'$.

\begin{pro}\label{pro:sym_finite_dist} The distance
$\de$
on
$\harm$
is symmetric,
$\de(q,q')=\de(q',q)$, $\de(q,q)=0$,
satisfies the triangle inequality,
$$\de(q,q'')\le\de(q,q')+\de(q',q''),$$
and finite
$\de(q,q')<\infty$,
for all
$q,q',q''\in\harm$. 
\end{pro}

\begin{proof} The property of
$\de$
to be symmetric and the triangle inequality immediately 
follows from the definition. Taking an empty zz-path, we see that
$\de(q,q)=0$
for any
$q\in\harm$.

The fact that the distance
$\de(q,q')$
is finite for every
$q$, $q'\in\harm$,
follows from Lemma~\ref{lem:zz_connected}.
\end{proof}

We similarly define the distance on
$\hm$,
for which we use the same notation
$\de$.
The canonical projection
$\harm\to\hm$
is a 2-sheeted covering of 3-manifolds with deck transformation group
isomorphic to
$\Z_2$
acting by 
$\de$-isometries. Then the distance on
$\harm$
is obtained by lifting the distance on
$\hm$.

A basic problem is to prove that 
$\de$
is nondegenerate, i.e.,
$\de(q,q')>0$
for any distinct
$q$, $q'\in\harm$.
It is not at all clear that this holds even in the case
$q$, $q'$
lie on a line, and moreover that
$\de(q,q')=|qq'|$
in this case.

\section{Projections to a line}
\label{sect:project_line}

\subsection{$s$-projection and midpoint projection}
\label{subsect:midpoint_project}

It follows from Lemma~\ref{lem:three_embeddings} that given
$a\in\ay$
and
$x\in X$, $x\notin a$,
there is a uniquely determined
$y\in X$
such that the pair
$(a,b)$
is harmonic, 
$(a,b)\in\hm$,
where
$b=(x,y)$.
In this case, we use notation
$x_a:=b$
and say that
$x_a\in\h_a$
is the projection of
$x$
to the line
$\h_a$.

We say that a one-parametric family of segments
$v_tw_t\sub\R$
is {\em monotone}, if its ends
$v_t$, $w_t$
are monotone in the same sence, i.e.,
$v_t<v_{t'}$
if and only if
$w_t<w_{t'}$
for
$t\neq t'$.

\begin{lem}\label{lem:monotone_segments} Given two lines
$\h_a$, $\h_c\sub\hm$
with
$a=(z,u)\in\ay$,
the family of segments
$v_aw_a=v_aw_a(p)\sub\h_a$
is monotone in
$p=(c,d)\in\h_c$, 
where
$d=(v,w)\in\ay$,
as
$p$
runs over the segment
$z_cu_c\sub\h_c$.
\end{lem}

\begin{proof} If
$c=a$,
then there is nothing to prove because
$z_cu_c=\h_c$ 
in this case and
$v_a=p=w_a$
for any
$p\in\h_c$.
Thus we assume that 
$c\neq a$.

Another trivial case occurs when the pair
$(a,c)$
is harmonic. In that case,
$z_c=u_c$,
i.e.
the segment
$z_cu_c$
is degenerate, and for 
$p=z_c=u_c$,
the family
$v_aw_a(p)=h_a$
is constant. Thus we assume that the pair
$(a,c)$
is not harmonic.

Let
$z'=\rho_c(z)$, $u'=\rho_c(u)$,
where
$\rho_c:X\to X$
is the reflection with respect to
$c$
(see the proof of Lemma~\ref{lem:common_perp} and \cite{Bu17}).
Then by definition
$z_c=(z,z')$, $u_c=(u,u')$.
Note that
$z_cu_c\sub\h_c$
is a ray when
$a$
and
$c$
have a common end.

Let
$zu\sub X$
be an open arc determined by
$z$, $u$
that does not contain at least one of the ends of
$c$,
and let
$z'u'\sub X$
be the 
$\rho_c$
image of
$zu$.
Then the ends
$v$, $w\in X$
of
$d$
miss the intersection
$zu\cap z'u'$
(which is nonempty if and only if the pairs
$a$, $c$
separate each other). We assume without loss of generality that
$v\in zu\sm z'u'$
($zu\sm z'u'\neq\es$
by the assumption that 
$(a,c)$
is not harmonic). Then
$w\in z'u'\sm zu$
because
$w=\rho_c(v)$.

Under our assumption, an order on
$z_cu_c$
induces well defined orders on the arcs
$zu\sm z'u'$, $z'u'\sm zu$
such that
$p<p'$
if and only if
$v<v'$
and
$w<w'$
for
$p=(c,d)$, $p'=(c,d')$, $d=(v,w)$, $d'=(v',w')$.
Taking projections on
$a$,
we see that the family
$v_aw_a(p)\sub\h_a$
is monotone in
$p$.
\end{proof}

We say that
$b\in\R$
is the 
$s$-point
of an (oriented) segment
$vw\sub\R$, $s>0$,
if
$b\in vw$
and
$|vb|/|bw|=s$.
For example,
$1$-point
is the midpoint of a segment
$vw$.
In that case, the order of
$v$, $w$
on
$\R$
is not important.

\begin{lem}\label{lem:midpoint_project} Given
$s>0$, $q=(a,b)\in\harm$, $a=(z,u)\in\ay$,
and a line 
$\h_c\sub\hm$,
there is a unique
$p=(c,d)\in z_cu_c\sub\h_c$, $d=(v,w)\in\ay$,
such that 
$b\in\h_a$
is the 
$s$-point 
of the (maybe degenerate) segment
$v_aw_a\sub\h_a$.
\end{lem}

\begin{proof} If
$c=a$,
then
$z_cu_c=\h_c$,
and for every
$p=(c,d)\in\h_c$
the segment
$v_aw_a$
is degenerate,
$v_a=w_a=p$.
In this case, we take
$p=(c,b)$.
We also do not exclude 
the case when the segment
$z_cu_c$
is degenerate, i.e.,
$z_c=u_c$.
In this case, the pair
$(a,c)$
is harmonic,
$d=a$
and
$v_aw_a=\h_a$.
Therefore, any
$b\in\h_a$
is understood as the 
$s$-point 
of the 
$\h_a$
ends at infinity.

As in the proof of Lemma~\ref{lem:monotone_segments}, we always assume that
$v\in zu\sm z'u'$
for
$d=(v,w)$, 
where
$z'=\rho_c(z)$, $u'=\rho_c(u)$.
Then
$w\in z'u'\sm zu$.
and we consider
$v_aw_a\sub\h_a$
as an oriented segment. This is well defined because by 
Lemma~\ref{lem:monotone_segments} the family
$v_aw_a=v_aw_a(p)$
in monotone in
$p\in z_cu_c$.

First, we show that any 
$b\in\h_a$
separates the 
$s$-points
$m_a'$ 
and
$m_a''$
of segments
$v_a'w_a'$, $v_a''w_a''$
respectively for appropriate
$d'=(v',w')$, $d''=(v'',w'')\in z_cu_c\sub\h_c$.
To this end, note that the 
$s$-point
$m_a$
of the segment
$v_aw_a\sub\h_a$
approaches the 
$\h_a$
ends at infinity
$z$
or
$u$
as
$d=(v,w)\in z_cu_c$
approaches
$z_c$
or
$u_c$
respectively. Indeed, one of
$v_a$, $w_a$
stays bounded on
$\h_a$
while the other one goes along
$\h_a$
to infinity as
$d\to z_c$
or
$u_c$.
Thus 
$m_a$
goes to
$z$
or
$u$
respectively. (It may happen that one of
$z_c$, $u_c\in\h_c$
is at infinity but not both of them when
$a$, $c\in\ay$
have a common end. In that case, the admissible segment
$z_cu_c$
is a ray on
$\h_c$,
and both
$v_a$, $w_a$
together with their 
$s$-point
$m_a$
go to respective end at infinity of
$\h_a$
when
$d=(v,w)\in\h_c$
goes to the infinite end of the ray).

Second, we conclude that any
$b\in\h_a$
is the 
$s$-point 
of a respective segment
$v_aw_a\sub\h_a$, $b=m_a$.
By the first part,
$b$
lies between
$m_a'$, $m_b''$.
Now, we move from
$d'$
to
$d''$
along
$\h_c$,
i.e. consider
$d_t=(1-t)d'+td''\in\h_c$, $0\le t\le 1$,
$d_t=(v_t,w_t)$.
The 
$s$-point
$(m_t)_a$
of
$(v_t)_a(w_t)_a$
varies continuously from
$m_a'$
to
$m_a''$
as
$t$
goes from 0 to 1. Therefore, there is
$0<\tau<1$
such that 
$b=(m_\tau)_a$.

Finally, we show that the required
$p=(c,d)\in\h_c$
is unique. Indeed, by Lemma~\ref{lem:monotone_segments}, segments
$v_aw_a$,
where
$d=(v,w)$,
are monotone in
$p\in z_cu_c$.
Thus for any other
$p'=(c,d')\in z_cu_c$
no one of the segments
$v_aw_a$, $v_a'w_a'$
contains the other one. Hence, the respective 
$s$-points
$m_a\neq m_a'$.
\end{proof}

For every
$s>0$
and a line
$\h_c\sub\hm$
Lemma~\ref{lem:midpoint_project} determines a map 
$\pr_c^s:\harm\to\h_c$,
which is called the 
$s$-{\em projection}
to the line
$\h_c$.
In the case
$s=1$
we abbreviate
$\pr_c:=\pr_c^1$,
and the map
$\pr_c:\harm\to\h_c$
is called the {\em midpoint} projection to
$\h_c$.
The map
$\pr_c^s\circ j:\harm\to\h_c$
in general differs from
$\pr_c^s$,
thus in
$\hm$
we have two maybe different projections to the line
$\h_c$
depending on the choice of one of the entries of
$q=(a,b)\in\hm$.
However,
$\pr_c^s$
is well defined along any line
$\h_a\sub\hm$,
and hence along any zz-path.

\subsection{Equal ratio projection}
\label{subsect:equal_ratio_project}

\begin{lem}\label{lem:equal_ratio_project} Given
$q=(a,b)\in\harm$, $a=(x,y)$, $b=(z,u)\in\ay$,
and a line
$\h_c\sub\hm$,
there is a unique
$p=(c,d)\in x_cy_c\cap z_cu_c\sub\h_c$, $d=(v,w)$,
such that 
$a\in v_bw_b\sub\h_b$, $b\in v_aw_a\sub\h_a$ 
and
$$\frac{|v_ba|}{|aw_b|}=\frac{|v_ab|}{|bw_a|}.$$
\end{lem}

\begin{proof} We fix an orientation of the line
$\h_c$
and the respective order. Note that one of the segments
$x_cy_c$, $z_cu_c$
is degenerate if and only if one of the pairs
$(a,c)$
or
$(c,b)$
is harmonic. Then there is nothing to prove because
$p=(c,a)$, $d=(v,w)=(x,y)$, $a=v_b=w_b\in\h_b$, $b\in v_aw_a=\h_a$ 
in the first case, and
$p=(c,b)$, $d=(v,w)=(z,u)$, $a\in v_bw_b=\h_b$, $b=v_a=w_a\in\h_a$,
in the second case (and the required equality is understood as
$0/0=\infty/\infty$).

Thus we assume that none of the segments
$x_cy_c$, $z_cu_c$
is degenerate. Moreover, their intersection
$x_cy_c\cap z_cu_c$
is not empty and also a nondegenerate segment because the pairs
$(x,y)$
and
$(z,u)$
being harmonic separate each other. Without loss of generality, we assume that
$x_c<y_c$, $u_c<z_c$.
Then
$u_c<y_c$
because the pairs
$(x,y)$
and
$(z,u)$
separate each other on
$X$.
Hence every
$d\in x_cy_c\cap z_cu_c$
separates
$u_c$, $y_c$
and
$x_c$, $z_c$.

Furthermore,
$c$
cannot separate
$(x,y)$, $(z,u)$.
Thus we can assume without loss of generality that
$c$
does not separate
$x$, $z$.
We also assume that
$v$
lies on the same arc in
$X$
determined by
$c$
as
$x$
and
$z$.
Then the assumption
$d=(v,w)\in x_cy_c\cap z_cu_c$
implies that 
$v$
lies between
$x$
and
$z$
on that arc.

It follows that when we are moving along
$\h_a$
from
$x$
to
$y$,
we meet
$v$ 
earlier than
$z$,
and
$u$
earlier than
$w$.
Hence,
$b\in v_aw_a$.
Similarly, when we are moving along
$\h_b$
from
$z$
to
$u$,
we meet
$v$
earlier than
$x$,
and
$y$
earlier than
$w$.
Hence
$a\in v_bw_b$.

To be definite we assume that
$x_cy_c\cap z_cu_c=x_cz_c$
(other cases are considered similarly). Thus if
$p=(c,d)\in x_cz_c$
goes to
$x_c$,
then
$v_a\to\infty$,
while
$w_a$
stays bounded on
$\h_a$, 
and
$v_b\to a$,
while
$\lim w_b\neq a$.
Setting
$s=|v_ba|/|aw_b|$, $t=|v_ab|/|bw_a|$,
we see that
$(s,t)\to (0,\infty)$
as
$p\to x_c$.
 
Similarly, if
$p\to z_c$,
then
$v_a\to b$,
while
$\lim w_a\neq b$, 
and
$v_b\to\infty$,
while
$w_b$
stays bounded. Therefore,
$(s,t)\to (\infty,0)$
in this case. By continuity, there is
$p\in x_cz_c$
with
$s=t$.
This gives a required
$p\in x_cy_c\cap z_cu_c\sub \h_c$.

By Lemma~\ref{lem:monotone_segments}, segments
$v_aw_a$, $v_bw_b$
are monotone in
$p\in x_cy_c\cap z_cu_c\sub\h_c$.
Since
$b\in v_aw_a$, $a\in v_bw_b$,
this implies that the rations
$s$, $t$
are monotone. Thus a required
$p$
is unique.
\end{proof}

For a line
$\h_c\sub\hm$,
Lemma~\ref{lem:equal_ratio_project} determines a map 
$\prr_c:\harm\to\h_c$,
which is called the {\em equal ratio projection} to the line
$\h_c$. 
Note that for 
$q=(a,b)\in\harm$
we have
$\prr_c(q)=\pr_c^s(q)$
for some well determined 
$s>0$,
where
$s$
depends on
$q$.

\subsection{Strictly contracting property of the midpoint projection}
\label{subsect:strictly_contacting}

Increment axiom (I) is only used in the proof of the following proposition
which plays a key role in the paper.

\begin{pro}\label{pro:strict_monotone} Given lines
$\h_a$, $\h_c\sub\hm$, $a\neq c$,
and points
$d=(v,w)$, $d'=(v',w')\in\h_c$
such that the pairs
$(v,w')$, $(v',w)$
separate each other, we have
$$\frac{1}{2}\left(|v_av_a'|+|w_aw_a'|\right)>|dd'|.$$
\end{pro}

For its proof see \cite[Proposition~7.11]{Bu17}.

\begin{lem}\label{lem:midpoint_project_estimate} The midpoint projection
$\pr_c:\h_a\to\h_c$
to any line
$\h_c\sub\hm$
is strictly contracting along any line
$\h_a\sub\hm$, $a\neq c$.
\end{lem}

\begin{proof} Given
$q=(a,b)$, $q'=(a,b')\in\h_a$,
we let
$p=\pr_c(q)$, $p'=\pr_c(q')$
be the midpoint projections to
$\h_c$.
Then
$p=(c,d)$, $p'=(c,d')$
with
$d=(v,w)$, $d'=(v',w')\in\ay$,
so that
$c$
and
$(v,w)$
separate each other as well as
$c$
and
$(v',w')$.
We assume without loss of generality that
$v$, $v'$
lie on an arc in
$X$
determined by
$c$,
while
$w$, $w'$
lie on the other arc determined by
$c$. 
Then the pairs
$(v,w')$, $(v',w)$
separate each other.

By Proposition~\ref{pro:strict_monotone}
$$\frac{1}{2}(|v_av_a'|+|w_aw_a'|)>|dd'|.$$

By definition of the midpoint projection,
$d$, $d'\in z_cu_c$,
where
$a=(z,u)$.
Thus
$a$
separates
$(v,v')$
and
$(w,w')$
(in terms of \cite{Bu17}, it means that the event
$(z,u)\in\ay$
is strictly between events 
$(v,v')$, $(w,w')\in\ay$).
By Lemma~\ref{lem:common_perp}, the pairs
$d=(v,w)$
and
$d'=(v',w')\in\h_c$
are in the strong causal relation. Since
$(z,u)$
separates 
$(v,v')$
and
$(w,w')$,
it follows that moving along
$a$,
we meet
$v,v'$
and
$w,w'$
in the same order. Identifying the line
$\h_a$
with real line
$\R$,
it means that the signs of
$v_a-v_a'$
and
$w_a-w_a'$
coincide.

Since
$b$
is the midpoint of
$v_aw_a$
and
$b'$
is the midpoint of
$v_a'w_a'$,
we obtain
\begin{align*}
|bb'|&=\left|\frac{1}{2}(v_a+w_a)-\frac{1}{2}(v_a'+w_a')\right|\\
     &=\frac{1}{2}|v_a-v_a'+w_a-w_a'|\\
     &=\frac{1}{2}(|v_av_a'|+|w_aw_a'|).
\end{align*}
Hence,
$|bb'|>|dd'|$
and therefore
$|qq'|>|pp'|$.
\end{proof}

\section{Distance $\de$ along segments}
\label{sect:distance_segments}

We assume that a monotone M\"obius structure
$M$
satisfies Increment axiom~(I), and under this assumption we prove
Theorem~\ref{thm:main}.

\begin{pro}\label{pro:sides_short} For any side 
$\si$
of a closed zz-path
$S\sub\harm$
we have 
$$|\si|<\sum_{\si'}|\si'|,$$
where the sum is taken over all sides
$\si'$
of
$S$, $\si'\neq\si$.
\end{pro}

\begin{proof} Let
$S'$
be a zz-path which is the union of all segments of
$S$
excluding
$\si$,
that is,
$S=\si\cup S'$.
The idea is to use the midpoint projection of
$S'$
to the line 
$h_c$
determined by
$\si$, $\si\sub h_c$,
and apply Lemma~\ref{lem:midpoint_project_estimate}. The problem is
that the midpoint projections on adjacent segments of
$S'$
may not coincide on the common vertex (in the case of the canonical
M\"obius structure on
$S^1$
they coincide). This could create gaps on
$\si$
which do not covered by the projection and thus prevent the required estimate. 
 
Let
$V=V(S')$
be the vertex set 
$S'$.
We fix 
$\ep>0$
and for every vertex
$v\in V$
we take
$\ep_v>0$
such that
$\sum_{v\in V}\ep_v<\ep$.
We use the midpoint projection on
$S'$
outside of the 
$\ep_v$-neighborhoods
$U_v(\ep_v)$, $v\in V$,
of vertices, the equal ratio projection on the vertices,
and interpolate between these types of projections inside of
$U_v(\ep_v)$
to obtain a continuous projection
$\pr_\si:S'\to\h_c$
with controlled metric properties.

Let
$\si'\sub S'$
be a side of
$S$
different from
$\si$, $h_a\sub \hm$
the line containing
$\si'$, $\si'=pq\sub h_a$.
If
$\si'$
is adjacent to
$\si$, 
then we assume to be definite that
$p\in\hm$
is the common vertex of
$\si$, $\si'$,
in particular,
$p=(a,c)$.
Then
$q=(a,c')$
for some
$c'\in\ay$.
In this case, by Lemmas~\ref{lem:midpoint_project}, \ref{lem:equal_ratio_project},
the whole segment
$\si'$
is projected to
$p$,
$\pr_\si(\si')=p$.

Thus we assume that
$\si'$
is not adjacent to
$\si$.
Let
$pq'\sub\si'$
be the minimal subsegment containing the 
$\ep_p$-neighborhood of
$p$
in
$\si'$, $|pq'|=\ep_p$.
We define
$\pr_\si$
on
$pq'$
by taking
$\pr_\si(p)=\prr_c(p)$, $\pr_\si(q')=\pr_c(q')$
and
$\pr_\si(p_\tau)=\pr_c^s(p_\tau)$,
where
$p_\tau=(1-\tau)p+\tau q'$,
$s=s(\tau)=(1-\tau)s_p+\tau$,
and
$s_p>0$
is determined by
$\prr_c(p)$, $\prr_c(p)=\pr_c^{s_p}(p)$
(we take
$s_p=1$ 
if the adjacent to
$\si'$
at
$p$
segment
$\si''\in S'$
is adjacent to
$\si$).

We have constructed a projection
$\pr_\si:S'\to\h_c$
of the zz-path
$S'$
to the line 
$\h_c$.
Continuity of
$\pr_\si$
along sides of
$S'$
follows from the uniqueness property of
$\pr_c^s$,
continuity at common vertices of adjacent sides
follows from definition of
$\prr_c$.
Since 
$\pr_\si$
is constant on the sides 
$\si_1$, $\si_2$
adjacent to
$\si$,
which are mapped to the vertices of
$\si$,
the continuity of
$\pr_\si$
implies that the image
$\pr_\si(S')\sub\h_a$
covers
$\si$.
Thus
$|\si|\le\sum_{\si'\neq\si}|\pr_\si(\si')|$. 

We decompose the right hand side of this inequality as
$\sum_{\si'\neq\si}|\pr_\si(\si')|=A+B$,
where
$A$
is the length of
$\pr_\si(S'\sm\cup_{v\in V}U_v(\ep_v))$
and
$B$
the length of
$\pr_\si(\cup_{v\in V}U_v(\ep_v))$.
Since
$\pr_\si$
coincides with the midpoint projection
$\pr_c$
on
$S'$
outside of the union
$\cup_{v\in V}U_v(\ep_v)$,
Lemma~\ref{lem:midpoint_project_estimate} gives
$A<|S''|-\sum_v\ep_v<|S''|$,
where
$S''=S'\sm(\si_1\cup\si_2)$.

Since
$\pr_\si$
is continuous on
$S'$,
we can make 
$B$
arbitrarily small taking
$\ep$
sufficiently small, say
$B<\de(\ep)<|\si_1|+|\si_2|$.
Thus
$|\si|\le|S''|+\de(\ep)<|S''|+|\si_1|+|\si_2|=|S'|$.
\end{proof}

\begin{cor}\label{cor:distance_line} For any
$q$, $q'\in\harm$
on a line we have
$\de(q,q')=|qq'|$.
\end{cor}

\begin{proof} Let
$S$
be a zz-path in
$\harm$
between
$q$, $q'$
different from the segment
$qq'$.
By definition, the first and last sides of
$S$
lie of the line determined by the segment
$qq'$.
We denote by
$\wt q$, $\wt q'$
the ends of the first and the last sides respectively,
and assume that we have the order
$q<\wt q<\wt q'<q'$
along the segment
$qq'$.
Any other order only makes arguments easier.

Let
$S'$
be the zz-subpath of
$S$
between
$\wt q$, $\wt q'$. 
Then
$S'$
together with the segment
$\wt q\wt q'$
gives a closed zz-path in
$\harm$.
By Proposition~\ref{pro:sides_short} we obtain
$|\wt q\wt q'|<|S'|$.
On the other hand,
$|qq'|=|\wt q\wt q'|+a$
and
$|S|=|S'|+a$,
where
$a=|q\wt q|+|\wt q'q'|$.
Hence
$|S|>|qq'|$
and thus
$\de(q,q')=|qq'|$. 
\end{proof}

This completes the proof of Theorem~\ref{thm:main}.

\end{document}